\documentclass[11pt]{article}
\usepackage{amsmath, amsthm, epsfig,amssymb, multicol,xcolor,tikz}
\usepackage[active]{srcltx}
\usepackage{fullpage}
\usepackage{hyperref}

\newcommand{\floor}[1]{{\lfloor #1 \rfloor}}
\newcommand{\ceil}[1]{{\lceil #1 \rceil}}
\newcommand{\F}{\mathcal{F}}
\newcommand{\G}{\mathcal{G}}
\newcommand{\HH}{\mathcal{H}}

\newcommand{\C}{\mathcal{C}}
\newcommand{\M}{{\cal M}}

\newcommand{\La}{{\rm La}}
\newcommand{\Lai}{{\rm La^{*}}}
\newcommand{\E}{{\rm E}}
\renewcommand{\Pr}{{\rm Pr}}
\renewcommand{\emptyset}{{\varnothing}}

\newcommand{\hb}{{\ell}}
\newcommand{\Lubell}{{\hb}}

\newcommand{\chain}[1]{K_{#1}}
\newcommand{\univ}[1]{U_{#1}}
\newcommand{\univd}[1]{U'_{#1}}
\newcommand{\stdex}[1]{S_{#1}}

\newcommand{\bool}[1]{2^{[#1]}}

\newcommand{\tth}{\pi^*}

\newcommand{\lth}{\lambda^*}

\newcommand{\slth}{\Lambda^*}
\newcommand{\kgmcomment}[1]{}

\newtheorem{lemma}{Lemma}
\newtheorem{theorem}{Theorem}
\newtheorem{cor}{Corollary}

\newtheorem{prop}{Proposition}
\newtheorem{conjecture}{Conjecture}

\newcommand{\st}{\colon\,}

\def\nchn{{\binom{n}{\lfloor n/2 \rfloor}}}

\def\tsupn{\bool{n}}

\title{Set families with forbidden subposets}
\author{Linyuan Lu
\thanks{University of South Carolina, Columbia, SC 29208,
({\tt lu@math.sc.edu}).  Research supported in part by NSF
grant DMS 1300547 and ONR grant N00014-13-1-0717.}
  \and Kevin G. Milans
\thanks{West Virginia University, Morgantown, WV 26505,
({\tt milans@math.wvu.edu}). } }

\begin{document}
\maketitle

\begin{abstract}
Let $\F$ be a family of subsets of $\{1,\ldots,n\}$.  We say that $\F$ is $P$-free if the inclusion order on $\F$ does not contain $P$ as an induced subposet.  The \emph{Tur\'an function} of $P$, denoted $\Lai(n,P)$, is the maximum size of a $P$-free family of subsets of $\{1,\ldots,n\}$.  We show that $\Lai(n,P) \le (4r + O(\sqrt{r}))\nchn$ if $P$ is an $r$-element poset of height at most $2$.  We also show that $\Lai(n,\stdex{r}) = (r+O(\sqrt{r}))\nchn$ where $\stdex{r}$ is the standard example on $2r$ elements, and that $\Lai(n,\bool{2}) \le (2.583+o(1))\nchn$, where $\bool{2}$ is the $2$-dimensional Boolean lattice.
\end{abstract}

\section{Introduction}

Tur\'an-type problems ask for the largest objects that do not contain a particular substructure.  We seek the largest set families whose inclusion order does not contain a copy of a fixed poset.  

A poset $P$ is a \emph{subposet} of $Q$ if $P\subseteq Q$ and $x \le_P y$ if and only if $x \le_Q y$.  We say that $Q$ \emph{contains a copy} of $P$ if $P$ is a subposet of $Q$.  If $P$ and $Q$ have the same elements and $x \le_P y$ implies that $x \le_Q y$, then $Q$ is an \emph{extension} of $P$.  An \emph{antichain} is a set of elements that are pairwise incomparable, and a \emph{chain} is a set of elements that are pairwise comparable.  Let $\chain{r}$ denote the $r$-element chain.  The \emph{height} of a poset $P$ is the maximum size of a chain in $P$, and the \emph{width} of $P$ is the maximum size of an antichain.

Let $\bool{n}$ denote the $n$-dimensional Boolean lattice; this is just the inclusion order on the subsets of $[n]$, where $[n] = \{1,\ldots,n\}$.  The \emph{levels} of the Boolean lattice are the families of the form $\binom{[n]}{k}$, where $\binom{S}{k}$ denotes the subsets of $S$ of size $k$.  If $\F\subseteq \tsupn$, then $\F$ is a subposet and also ordered by containment.  If $\F$ does not contain a copy of $P$, we say that $\F$ is \emph{$P$-free}.  The \emph{Tur\'an function} of $P$, denoted $\Lai(n,P)$, is the maximum size of a $P$-free family $\F$ with $\F \subseteq \tsupn$.  In 1928, Sperner~\cite{Sperner} proved that the maximum size of an antichain in $\bool{n}$ is $\nchn$.  In our language, Sperner's theorem states that $\Lai(n,\chain{2}) = \nchn$.  Erd\H{o}s~\cite{Erdos} extended Sperner's result to longer chains, showing that $\Lai(n,\chain{r})$ is the sum of the $r-1$ largest binomial coefficients in $\{\binom{n}{0}, \binom{n}{1}, \ldots, \binom{n}{n}\}$.  Hence $\Lai(n,\chain{r}) = (r-1 + o(1))\nchn$ for fixed $r$ and growing $n$.  Except when $P$ is an antichain, $\Lai(n, P)$ is asymptotic to a positive integer multiple of $\nchn$ in all known cases.  It is convenient to define the \emph{Tur\'an threshold}, denoted $\tth(P)$, to be $\limsup_{n\to\infty} \Lai(n,P)/\nchn$ whenever the latter is finite.  Clearly, $\Lai(n, P) \le (\tth(P) + o(1))\nchn$ for each poset $P$.  Erd\H{o}s's result implies that $\tth(\chain{r}) = r-1$.

A well-studied variant of this problem uses a weaker notion of containment.  The \emph{weak Tur\'an function} of $P$, denoted $\La(n,P)$, is the maximum size of a family of subsets $\F$ of $[n]$ such that $\F$ does not contain an extension of $P$.  Under this weaker notion of containment, one must avoid $P$ as well all proper extensions of $P$, and so always $\La(n,P) \le \Lai(n,P)$.  For chains, the two notions of containment coincide and equality holds.  Many groups obtained bounds on the weak Tur\'an function for specialized poset families; see~\cite{ext1, ext2, ext3, ext4, ext5, ext6} for results spanning nearly 25 years.  In 2009, Bukh~\cite{Bukh} unified many previous results by showing that $\La(n,P) \le (k-1)\nchn(1+O(1/n))$ when $P$ is tree poset of height $k$.  Perhaps the most celebrated special case is the forbidden diamond problem, where $P=\bool{2}$.  Axenovich, Manske, and Martin~\cite{AMM} proved that $\La(n,\bool{2})\le (2.284+o(1))\nchn$.  Griggs, Li, and Lu~\cite{GLL} reduced the constant to $2+3/11$, which is approximately $2.273$.  The current record of $2.25$ is due to Kramer, Martin, and Young~\cite{KMY}, and it appears that new tools will be required to improve the bound further.  The two largest levels show that $\La(n,\bool{2})\ge (2-o(1))\nchn$, and it is conjectured that $\La(n,\bool{2})$ is asymptotic to $2\nchn$.  

Aside from the classic results of Sperner and Erd\H{o}s, not much is known about Tur\'an thresholds in the induced case.  In 2008, Carroll and Katona~\cite{CK} proved that 
\[ (1 + 1/n + \Omega(1/n^2))\nchn \le \Lai(n, V_2) \le (1 + 2/n + O(1/n^2)\nchn, \]
where $V_2$ is the poset on elements $\{a, b_1, b_2\}$ with $a\le b_1$ and $a\le b_2$, and no other relations.  It follows that $\tth(V_2) = 1$.  A \emph{tree poset} is a poset whose Hasse diagram is a tree.  In 2012, Boehnlein and Jiang~\cite{BJ} gave an induced analogue of Bukh's result.  They showed that $\Lai(n, P) \le (k-1 + O(\sqrt{\ln n/n})\nchn$ when $P$ is a tree poset of height $k$.  Since the largest $k-1$ levels of $\bool{n}$ are $P$-free when $P$ has height $k$, it follows that $\tth(P) = k-1$. 

In this article, we show that series-parallel posets and posets of height 2 have finite Tur\'an thresholds.  We believe much more is true.
\begin{conjecture}\label{conj:tt}
Every poset has a finite Tur\'an threshold.
\end{conjecture}
Conjecture~\ref{conj:tt} is equivalent to the statement that for each poset $P$, the Tur\'an function satisfies $\Lai(n,P) = O(\nchn) = O(2^n/\sqrt{n})$.  It is known (see \cite{JLM}) that for each poset $P$, there exists $\alpha > 0$ such that $\Lai(n,P) = O(2^n/n^\alpha)$.  

In general, proving upper bounds on $\Lai(n,P)$ appears to be more challenging than proving upper bounds on $\La(n,P)$.  For example, the extension-containment analogue of Conjecture~\ref{conj:tt} follows directly from Erd\H{o}s's result, since $\La(n,P) \le \La(n,\chain{r}) \le (r-1+o(1))\nchn$ when $P$ is an $r$-element poset; see~\cite{BN} and~\cite{CL} for improvements for general $P$.

To establish bounds on the Tur\'an function and the Tur\'an threshold, it is convenient to use a weighting on $\tsupn$ in which the set $A\in\tsupn$ has weight $1/\binom{n}{|A|}$.  This weighting scheme has a nice probabilistic interpretation.  A maximal chain in $\tsupn$ contains one set of each size in $\{0,\ldots,n\}$.  If a maximal chain is chosen uniformly at random, then all sets of a particular size are equally likely to appear in the chain.  The weight of $A\in \tsupn$ is equal to the probability that a randomly chosen maximal chain in $\tsupn$ contains $A$.  

When $\F \subseteq \tsupn$, we define the \emph{Lubell function} of $\F$, denoted $\hb(\F)$, to be the expected number of times that a randomly chosen maximal chain in $\bool{n}$ meets $\F$.  Since $\hb(\F) = \sum_{A\in\F} \frac{1}{\binom{n}{|A|}}$, the Lubell function may be viewed as an alternative measure of the size of $\F$ where sets in the middle of the Boolean lattice have small weight and sets and the ends have large weight.  The value of the Lubell function at $\F$ is the \emph{Lubell mass} of $\F$.  We define the \emph{Lubell threshold}, denoted $\lth(P)$, to be $\limsup_{n\to \infty} \{\hb(\F)\st \mbox{$\F\subseteq \tsupn$ and $\F$ is $P$-free}\}$.  Since $\hb(\F) \ge |\F|/\nchn$, it follows that $\lth(P) \ge \tth(P)$.  For most of our results on the Lubell mass of $P$-free families, we do not need $n$ to be large.  In these cases, we use the \emph{strong Lubell threshold}, denoted $\slth(P)$, which is simply $\sup \{\hb(\F) \st \mbox{$\F$ is $P$-free}\}$.  Clearly, $\slth(P) \ge \lth(P)$ for all $P$, and if $\lth(P)$ is finite, then so is $\slth(P)$.  In our language, Lubell gave an elegant proof of Erd\H{o}s's result by showing that $\slth(\chain{r}) = r-1$.  To establish our results, we prove that certain posets have finite Lubell thresholds.  We propose a strengthening of Conjecture~\ref{conj:tt}.

\begin{conjecture}\label{conj:lt}
Every poset has a finite Lubell threshold.
\end{conjecture}

In Section~\ref{sec:sp}, we prove that series-parallel posets have finite Lubell thresholds.  In Section~\ref{sec:bip}, we prove that posets of height at most $2$ have finite Lubell thresholds.  We also show that $r-2 \le \tth(\stdex{r}) \le \slth(\stdex{r}) = r + O(\sqrt{r})$, where $\stdex{r}$ denotes the standard example on $2r$ elements.  In Section~\ref{sec:dia}, we establish a preliminary upper bound for the induced forbidden diamond, showing that $\tth(\bool{2}) < 2.583$.

\section{Series-Parallel Posets}\label{sec:sp}

Given posets $P$ and $Q$, the \emph{series construction with $P$ below $Q$} forms a new poset by taking the disjoint union of $P$ and $Q$ and setting each element of $P$ to be less than each element of $Q$.  Similarly, the \emph{parallel construction} forms a new poset by taking the disjoint union of $P$ and $Q$ in which each element of $P$ is incomparable with each element of $Q$.  The family of \emph{series-parallel posets} are defined inductively as follows.  The single element poset is a series-parallel poset.  If $P$ and $Q$ are series-parallel posets, then the posets formed via the series construction and the parallel construction are also series-parallel posets. 

To establish these results, we need a refined notion of the Lubell function.  In a poset $P$ with $x\le y$, the \emph{interval $[x,y]$} is $\{z\in P\st x\le z \le y\}$.  When $\F\subseteq \tsupn$ and $I$ is an interval in $\tsupn$, we define $\hb(\F; I)$ to be the expected number of times that a random, maximal chain in $I$ meets $\F$.  It is convenient to define $\hb_A^-(\F)$ to be $\hb(\F; [\emptyset, A])$ and $\hb_A^+(\F)$ to be $\hb(\F; [A, [n]])$.

\begin{prop}\label{prop:UD}
Let $\F$ be a nonempty family of subsets of $[n]$ and let $p$ be the probability that a random maximal chain from $\emptyset$ to $[n]$ meets at least one set in $\F$.  There exists $A, A' \in F$ such that $\hb_A^+(\F) \ge \hb(\F)/p$ and $\hb_{A'}^-(\F) \ge \hb(\F)/p$.
\end{prop}
\begin{proof}
Let $X$ be the number of times that a random maximal chain $C$ from $\emptyset$ to $[n]$ meets $\F$, and for each $A\in \F$, let $E_A$ be the event that $\min \{C\cap \F\} = A$, and let $E_0$ be the event that $C \cap \F = \emptyset$.  Note that $E_0$ and the events in $\{E_A\st A\in \F\}$ partition the space of all maximal chains.  Therefore 
\begin{align*}
\hb(\F) = \E[X] &= \E[X|E_0]\cdot\Pr[E_0] + \sum_{A\in \F} \E[X|E_A]\cdot\Pr[E_A] \\
	& = 0\cdot(1-p) + \sum_{A \in \F} \hb_A^+(\F)\cdot\Pr[E_A] \\
	&\le \max_{A\in \F} \{\hb_A^+(\F)\} \sum_{A\in \F} \Pr[E_A] \\
	&= \max_{A\in \F} \{\hb_A^+(\F)\} p.
\end{align*}
The second statement follows by symmetry.
\end{proof}

A family $\F \subseteq \tsupn$ is \emph{upwardly $\alpha$-shallow} if $\hb_A^+(\F) \le \alpha$ for each $A\in\F$.  Similarly, $\F$ is \emph{downwardly $\alpha$-shallow} if $\hb_A^-(\F) \le \alpha$ for each $A\in \F$.  We say that $\F$ is \emph{$\alpha$-shallow} if it is upwardly $\alpha$-shallow or downwardly $\alpha$-shallow.

\begin{cor}\label{cor:shallow}
If $\F$ is $\alpha$-shallow, then $\hb(\F) \le \alpha$.
\end{cor}
\begin{proof}
By symmetry, we may assume that $\F$ is upwardly $\alpha$-shallow.  By Proposition~\ref{prop:UD}, there exists $A\in\F$ such that $\hb_A^+(\F) \ge \hb(\F)$.  Since $\F$ is upwardly $\alpha$-shallow, it must be that $\hb_A^+(\F) \le \alpha$.  
\end{proof}

If $x$ is an element in $P$, we define the \emph{open downset} of $x$, denoted $D_P(x)$, to be $\{y\in P\st y<x\}$ and the \emph{closed downset} of $x$, denoted $D_P[x]$, to be $D(x)\cup \{x\}$.  Similarly, we denote the \emph{open upset} of $x$ by $U_P(x)$ and \emph{closed upset} of $x$ by $U_P[x]$.  In each of these, we may omit the subscript when $P$ is clear from context. 

\begin{lemma}\label{lem:series}
Let $P_1$ and $P_2$ be posets and let $P$ be the poset obtained by introducing a new element $u$, setting all elements of $P_1$ to be below $u$, and all elements of $P_2$ to be above $u$.  If $\slth(P_1)$ and $\slth(P_2)$ are finite, then so is $\slth(P)$ and  $\slth(P) \le \slth(P_1) + \slth(P_2) + 2$.
\end{lemma}
\begin{proof}
Let $\alpha_1= \slth(P_1)$, $\alpha_2 = \slth(P_2)$, and let $\F \subseteq \tsupn$ satisfy $\hb(\F) > \alpha_1 + \alpha_2 + 2$.  We show that $\F$ contains a copy of $P$.  Let $\F_1 = \{A\in\F\st \hb_A^-(\F) \le \alpha_1 + 1\}$, and note that $\F_1$ is a downwardly $(\alpha_1 + 1)$-shallow set.  By Corollary~\ref{cor:shallow}, we have that $\hb(\F_1)\le \alpha_1 + 1$.  Similarly, let $\F_2 = \{A\in\F\st \hb_A^+(\F) \le \alpha_2 + 1\}$, and note that Corollary~\ref{cor:shallow} implies that $\hb(\F_2) \le \alpha_2 + 1$. 

Let $\F' = \F - (\F_1 \cup \F_2)$.  We have that $\hb(\F') \ge \hb(\F) - \hb(\F_1) - \hb(\F_2) > (\alpha_1 + \alpha_2 + 2) - (\alpha_1 + 1) - (\alpha_2 + 1) = 0$.  It follows that $\F'$ is nonempty.  Let $A\in \F'$; we obtain a copy of $P$ in which $A$ is identified with $u$.

Since $A\not\in \F_1$, we have that $\hb_A^-(\F) > \alpha_1 + 1$ and hence $\hb_A^-(\F - \{A\}) = \hb_A^-(\F) - 1 > \alpha_1$.  Since $\slth(P_1) = \alpha_1$, it follows that $D_\F(A)$ contains a copy of $P_1$.  Similarly, $U_\F(A)$ contains a copy of $P_2$.
\end{proof}

Note that Lemma~\ref{lem:series} is sharp when $P_1$ and $P_2$ are chains.  We say that $\F\subseteq \tsupn$ is \emph{balanced} if $\hb(\F;I) \le \hb(\F)$ for each interval $I$ in $\tsupn$.  When we wish to find induced subposets of $\F$ under the assumption $\hb(\F) >\alpha$, we may assume that $\F$ is balanced.  Indeed, for a general family $\F \subseteq \tsupn$, choose an interval $[A,B]$ to maximize $\hb(\F; [A,B])$, and let $\F' = \{C-A\st C\in\F \cap [A,B]\}$.  Note that $\F'\subseteq 2^{B-A}$, and relative to this subcube, $\F'$ has Lubell mass $\hb(\F, [A,B])$, which is at least $\hb(\F)$.  By the selection of $[A,B]$, the family $\F'$ is balanced relative to $2^{B-A}$.

For disjoint sets $A,B\subseteq [n]$, we define $R_A^B$ to be the collection of all sets in $\tsupn$ that contain $A$ but are disjoint from $B$.  Note that $R_A^B$ is just the interval $[A, [n] - B]$ and is hence a subcube of dimension $n-|A|-|B|$.  When $A$ or $B$ is small, we may omit the set notation and write $R_{ij}^k$ for $R_{\{i,j\}}^{\{k\}}$ or simply $R_i$ for $R_{\{i\}}^\emptyset$.

\begin{lemma}\label{lem:para}
Let $P_1$ and $P_2$ be posets and let $P$ be the disjoint union of $P_1$ and $P_2$ with each element of $P_1$ incomparable with each element of $P_2$.  If $\slth(P_1)$ and $\slth(P_2)$ are well defined, then $\slth(P) \le \max\{\slth(P_1), \slth(P_2)\} + 8$.
\end{lemma}
\begin{proof}
Let $\alpha = \max\{\slth(P_1), \slth(P_2)\}$.  Let $\F$ be a family of subsets of $[n]$ with $\hb(\F)>\alpha + 8$; we may assume without loss of generality that $\F$ is balanced.  We show that for some pair $\{i,j\}\in\binom{[n]}{2}$, both $\hb(\F; R_i^j)$ and $\hb(\F; R_j^i)$ are larger than $\alpha$.  Since $z_1$ and $z_2$ are incomparable when $z_1 \in R_i^j$ and $z_2 \in R_j^i$, we obtain an induced copy of $P$ with elements in $P_1$ represented by elements in $\F \cap R_i^j$ and elements in $P_2$ represented by elements in $\F \cap R_j^i$.

For $\{i,j\}\in\binom{[n]}{2}$, let $\theta_{ij} = \hb(\F; R_{ij}) + \hb(\F; R_i^j) + \hb(\F; R_j^i) + \hb(\F; R^{ij})$, and let $\theta = \sum_{\{i,j\}\in\binom{[n]}{2}} \theta_{ij}$.  Note that $\{R_{ij}, R_i^j, R_j^i, R^{ij}\}$ is a partition of $\tsupn$ and therefore each $A\in\F$ contributes to exactly one term in $\theta_{ij}$.

Consider $A\in\F$, and let $k=|A|$.  We compute the contribution of $A$ to $\theta$.  There are $\binom{k}{2}$ pairs $\{i,j\}$ such that $A\in R_{ij}$, and each such term contributes $1/\binom{n-2}{k-2}$ to $\theta$.  There $k(n-k)$ pairs $\{i,j\}$ such that $A\in R_i^j\cup R_j^i$, and each such term contributes $1/\binom{n-2}{k-1}$ to $\theta$.  There are $\binom{n-k}{2}$ pairs $\{i,j\}$ such that $A\in R^{ij}$ and each such term contributes $1/\binom{n-2}{k}$ to $\theta$.  Adding these, we find that the contribution of $A$ to $\theta$ is 
\[ \binom{k}{2}\cdot \frac{1}{\binom{n-2}{k-2}} + k(n-k)\cdot \frac{1}{\binom{n-2}{k-1}} + \binom{n-k}{2}\cdot \frac{1}{\binom{n-2}{k}}.\]
When $2\le k\le n-2$, this readily simplifies to $4\binom{n}{2}\cdot 1/\binom{n}{k}$, which is $4\binom{n}{2}$ times the contribution of $A$ to $\hb(\F)$.  When $k$ is outside this range, some of the coefficients are zero (and the corresponding contributions are undefined).  If $k\in\{1,n-1\}$, the sum simplifies to $3\binom{n}{2}\cdot 1/\binom{n}{1}$, which is $3\binom{n}{2}$ times the contribution of $A$ to $\hb(\F)$.  Finally, when $k\in\{0,n\}$, the sum simplifies to $\binom{n}{2}\cdot 1/\binom{n}{0}$, which is $\binom{n}{2}$ times the contribution of $A$ to $\hb(\F)$.  Since $\{A\in\F \st |A|\in\{0,n\}\}$ contributes at most $2$ to $\hb(\F)$ and $\{A\in\F \st |A|\in\{1,n-1\}\}$ also contributes at most $2$ to $\hb(\F)$, it follows that 
\begin{align*}
\theta &\ge 2\cdot \binom{n}{2} + 2\cdot 3\binom{n}{2} + (\hb(\F) - 4)\cdot 4\binom{n}{2} \\
	& = \binom{n}{2}(4\hb(\F) - 8).
\end{align*}
Choose a pair $\{i,j\}$ so that $\theta_{ij} \ge 4\hb(\F) - 8$.  Because $\F$ is balanced, we have that each of the 4 terms in the sum defining $\theta_{ij}$ is at most $\hb(\F)$.  It follows that each term is at least $\hb(\F) - 8$, and $\hb(\F) - 8 > \alpha$.
\end{proof}

\section{Posets of Height 2}\label{sec:bip}

In this section, we show that Conjecture~\ref{conj:lt} holds for posets of height $2$.  To do so, we need a special family of posets.  Let $\univ{r}$ be the poset with antichains $\{a_1, \ldots, a_r\}$ and $\{b_S\st S\subseteq [r]\}$ where $a_j \le b_S$ if and only if $j\in S$.  The \emph{dual} of a poset $P$ is the poset $Q$ on the same elements with $x \le_Q y$ if and only if $y \le_P x$.  Let $\univd{r}$ be the dual of $\univ{r}$; this poset has the same elements but now $b_S \le a_j$ if and only if $j\in S$.

\tikzstyle vertex=[circle,fill=black!35,inner sep=1pt]
\tikzstyle edge=[draw,line width=1pt]

\begin{figure}
    \begin{center}
		\begin{tikzpicture}[xscale=1.25]
			\foreach \x/\n in {-1/1, 0/2, 1/3}
			{
				\node[vertex] at (\x, 1) (E\n) {} ;
				\node[anchor=south] at (E\n) {$a_\n$} ;
			}

			\begin{scope}[xshift={-3.5cm}, every path/.style={edge}]
				\node[vertex] at (0, 0) (A0) {} ;

				\node[vertex] at (1, 0) (A1) {} ;
				\draw (A1) -- (E1) ;

				\node[vertex] at (2, 0) (A2) {} ;
				\draw (A2) -- (E2) ;

				\node[vertex] at (3, 0) (A3) {} ;
				\draw (A3) -- (E3) ;

				\node[vertex] at (4, 0) (A12) {} ;
				\draw (A12) -- (E1) ;
				\draw (A12) -- (E2) ;

				\node[vertex] at (5, 0) (A13) {} ;
				\draw (A13) -- (E1) ;
				\draw (A13) -- (E3) ;

				\node[vertex] at (6, 0) (A23) {} ;
				\draw (A23) -- (E2) ;
				\draw (A23) -- (E3) ;

				\node[vertex] at (7, 0) (A123) {} ;
				\draw (A123) -- (E1) ;
				\draw (A123) -- (E2) ;
				\draw (A123) -- (E3) ;

                \foreach \n/\txt in {
				    A0/\emptyset, A1/\{1\}, A2/\{2\}, A3/\{3\},
				    A12/{\{1,2\}}, A13/{\{1,3\}}, A23/{\{2,3\}}, A123/{\{1,2,3\}}
				    }
				{{
				    \path (\n) +(0, -2ex) node {$b_{\txt}$} ;
				}}
			\end{scope}
		\end{tikzpicture}
	\end{center}
    \caption {$\univd{3}$}\label{fig:univ}
\end{figure}   

Next, we show that both $\univ{r}$ and $\univd{r}$ are universal for posets of height $2$.

\begin{lemma}\label{lem:univP}
If $P$ is an $r$-element poset of height at most $2$, then $P$ is a subposet of both $\univ{r}$ and $\univd{r}$.
\end{lemma}
\begin{proof}
It suffices to show that $P$ is a subposet of $\univd{r}$, since $P' \subseteq \univd{r}$ implies that $P \subseteq \univ{r}$, where $P'$ is the dual of $P$.  Let $x_1, \ldots, x_r$ be the elements of $P$, indexed so that the set of maximal elements of $P$ is $\{x_1, \ldots, x_m\}$.  For $1\le i\le m$, we identify $x_i$ in $P$ with $a_i$ in $\univd{r}$.  For $m+1 \le i \le r$, let $S_i = \{j\in[m]\st x_i \le x_j\} \cup \{i\}$, and we identify $x_i$ in $P$ with $b_{S_i}$ in $\univd{r}$.  For $m+1 \le i, j \le r$, we have that $i\in S_j$ if and only if $i=j$, and therefore the sets $S_{m+1}, \ldots, S_r$ are distinct.
\end{proof}

By Lemma~\ref{lem:univP}, to establish Conjecture~\ref{conj:lt} for posets of height 2, it suffices to show that each family with sufficiently large Lubell mass contains $\univ{r}$ or $\univd{r}$.  The following concept is key.  Let $\F\subseteq \tsupn$ and let $A\in \F$.  An element $i\in A$ is a \emph{pivot} of $A$ if there exists $j\not\in A$ such that $A - \{i\} \cup \{j\} \in \F$.  We say that $A$ is \emph{$\gamma$-flexible} if it has at least $\gamma |A|$ pivots.  As it turns out, if each set in $\F$ is not too large and $\F$ has no $\gamma$-flexible sets, then the Lubell function of $\F$ is bounded.

\begin{lemma}\label{lem:flex}
Let $\gamma$ and $\delta$ be a real numbers in the range $[0,1)$ and let $\F$ be a family of subsets of $[n]$ such that $|A| \le \delta n$ for each $A\in\F$.  If $\F$ does not contain a $\gamma$-flexible set, then $\hb(\F) < 1 + \frac{1}{1-\gamma}\ln \frac{1}{1-\delta}$.
\end{lemma}
\begin{proof}
Let $L_k = \F \cap \binom{[n]}{k}$.  If $A\in L_k$ and $i\in A$ but $A - B\ne \{i\}$ for each $B\in L_k$, then $A$ is the only set in $L_k$ that contains $A - \{i\}$.  Since no set in $\F$ is $\gamma$-flexible, for each $A\in L_k$, there are at least $(1-\gamma)k$ indices $i\in A$ such that $A$ is the only set in $L_k$ containing the $(k-1)$-set $A - \{i\}$.  It follows that $|L_k|\cdot (1-\gamma)k \le \binom{n}{k-1}$.  Therefore
\begin{align*}
\hb(\F) &= \sum_{0 \le k \le \delta n} |L_k| \cdot \frac{1}{\binom{n}{k}} \\
	&\le 1 + \sum_{1 \le k \le \delta n} |L_k| \cdot \frac{1}{\binom{n}{k}} \\
	&\le 1 + \sum_{1 \le k \le \delta n} \frac{1}{1-\gamma} \cdot k \frac{\binom{n}{k-1}}{\binom{n}{k}} \\
	&=   1 + \frac{1}{1-\gamma} \sum_{1 \le k \le \delta n} \frac{1}{n-k+1}\\
	&<   1 + \frac{1}{1-\gamma} \ln \frac{1}{1-\delta}.
\end{align*}
\end{proof}

In most applications of Lemma~\ref{lem:flex}, we set $\delta =1/2$.  Let $\F \subseteq\tsupn$.  For a set of indices $T$, the \emph{projection} of $\F$ onto $T$ is the family $\{ A \cap T\st A\in \F\}$.  To find a copy of $\univd{r}$ in $\F$, we first find a $\gamma$-flexible set $A$ such that $\hb_A^-(\F)$ is large.  Then, we will find a particular structure in the projection of $\F$ onto $T$.  To force this structure, we will need the projection to have large Lubell mass.

\begin{lemma}\label{lem:proj}
Let $\F$ be a family of subsets of $[n]$, let $T\subseteq [n]$ with $|T| = t$, and let $\F'$ be the projection of $\F$ onto $T$.  We have that $\hb_T^-(\F') \ge \frac{t+1}{n+1} \hb(\F)$.
\end{lemma}
\begin{proof}
We show that $\hb(\F) \le \hb_T^-(\F') \frac{n+1}{t+1}$.  For each $B\subseteq T$, let $\G_B = \{A\subseteq [n]\st A\cap T = B\}$.  Note that $B\in \F'$ if and only if $\G_B \cap \F$ is non-empty, and when $\G_B \cap \F$ is non-empty, adding additional members of $\G_B$ to $\F$ does not change $\F'$.  Hence we may assume that $\F$ is the union of some of the families in $\{\G_B\st B\subseteq T\}$.  

Next, we compute $\hb(\G_B)$.  Let $|B| = k$, let $X$ be the number of times that a random full chain meets $\G_B$, and consider the corresponding permutation $\sigma$ of $[n]$.  Let $\sigma'$ be the restriction of $\sigma$ to $T$.  Note that $\sigma$ determines $X$ as follows.  If the first $k$ elements of $\sigma'$ are not exactly the elements of $B$, then $X=0$.  So, consider the case where the first $k$ elements of $\sigma'$ are the elements of $B$.  The $k$th element in $\sigma'$ contributes $1$ to $X$ and additional contributions arise as follows.  The elements of $T$ partition $[n] - T$ into $t+1$ intervals in $\sigma$, and each element in $[n] - T$ between the $k$th element of $\sigma'$ and the $(k+1)$st element of $\sigma'$ contributes $1$ to $X$.  Since each element in $[n] - T$ is equally likely to occur in any of the $t+1$ intervals, the expected size of each interval is $(n-t)/(t+1)$. Therefore
\[ \hb(\G_B) = \frac{1}{\binom{t}{k}}\left(1+\frac{n-t}{t+1}\right) = \frac{1}{\binom{t}{k}}\frac{n+1}{t+1}.\]

It follows that 
\[ \hb(\F) \le \sum_{B\in\F'} \hb(\G_B) = \frac{n+1}{t+1} \sum_{B\in\F'} \frac{1}{\binom{t}{|B|}}  = \frac{n+1}{t+1}\hb_T^-(\F'). \]
\end{proof}

\newcommand{\vcdim}{\mathrm{VCdim}}
For $R\subseteq [n]$, the family $\F$ \emph{shatters} $R$ if the projection of $\F$ onto $R$ is $R$.  The \emph{VC-dimension} of $\F$, denoted $\vcdim(\F)$, is the maximum size of a shattered subset of $[n]$.  The following classical lemma has been known since the introduction of the VC-dimension (see~\cite{VC, Sauer, Shelah}). 

\begin{lemma}[Shatter Function Lemma]\label{lem:VCdimS}
If $\F \subseteq \tsupn$ and $\vcdim(\F) < d$, then $|\F| \le \sum_{k=0}^{d-1} \binom{n}{k}$.
\end{lemma}

The bound in Lemma~\ref{lem:VCdimS} is best possible, since $\{A\subseteq[n]\st |A| \le d-1\}$ has VC-dimension $d-1$.  Obtaining the largest Lubell mass of a family $\F$ with $\vcdim(\F) < d$ is more subtle.  Since $|\F| \le \sum_{k=0}^{d-1} \binom{n}{k}$ implies that $\hb(\F) < 2d$, we obtain the following corollary.

\begin{cor}\label{cor:VCdimL}
If $\F \subseteq \tsupn$ and $\vcdim(\F) < d$, then $\hb(\F) < 2d$.
\end{cor}

Perhaps surprisingly, the bound in Corollary~\ref{cor:VCdimL} is essentially best possible.  For a positive integer $t$, partition $[n]$ into sets $Z_1, \ldots, Z_t$ of nearly equal size.  Let $\G = \{A\subseteq [n]\st |A| > n - \floor{(1-1/t)d}\}$, let $\HH$ be the family of all sets $A\subseteq [n]$ such that $|A \cap Z_i| < d/t$ for all $i$, and let $\F = \G \cup \HH$.  First, we claim that $\vcdim(\F) < d$.  Indeed, if $R \in\binom{[n]}{d}$, then for some part $Z_i$, we have that $|R \cap Z_i| \ge d/t$.  Let $S$ be a subset of $R \cap Z_i$ of size $\ceil{d/t}$.  We show that if $A\in \F$, then $A \cap R \ne S$, which implies that $\F$ does not shatter $R$.  Suppose that $A \subseteq [n]$ and $A \cap R = S$.  We show that $A \not\in \F$.  Indeed, since $A \cap R = S$, we have that $S \subseteq A$.  Hence $|A \cap Z_i| \ge |S \cap Z_i| \ge d/t$, and so $A\not\in \HH$.  On the other hand, $A$ is disjoint from $R - S$, and $|R - S| \ge d - \ceil{d/t}$.  Therefore $|A| \le n - (d - \ceil{d/t}) = n - \floor{(1-1/t)d}$ and hence also $A\not\in \G$.  Hence, $\vcdim(\F) < d$.

Note that $\hb(\G) = \floor{(1-1/t)d}$ and that $\hb(\HH) = \sum_k p_k$, where $p_k$ is the probability that a random $k$-set is in $\HH$.  For $k<d/t$, we have that $p_k = 1$.  Also, $p_k \ge 1 - tq_k$, where $q_k$ is the probability that a random $k$-set meets a fixed set $U$ of size $\ceil{n/t}$ in at least $d/t$ points.  Here, we need a concentration inequality.

\begin{lemma}\label{lem:ci}
Let $U\subseteq [n]$, let $\gamma = |U|/n$, and let $\delta \ge 0$.  If $A$ is a random $k$-set and $X = |A \cap U|$, then 
\[ \Pr[X \ge (1+\delta)\gamma k] \le e^{-\frac{(\delta\gamma)^2}{2} k}. \]
\end{lemma}
\begin{proof}
We select $A$ by iteratively choosing elements $a_1, \ldots, a_k$ from those in $[n]$ not already selected.  Define $X_i$ to be the random variable $\E[X | a_1, \ldots, a_i]$.  Note that $X_k = |A \cap U| = X$, $X_0 = \E[X] = \gamma k$, and $X_0, \ldots, X_k$ is a martingale.  Also, since altering any one of $a_1, \ldots, a_k$ changes $|A \cap U|$ by at most $1$, we have $|X_i - X_{i-1}| \le 1$.  The lemma now follows from Azuma's inequality.
\end{proof}

Suppose that $n \ge 2t^2$ and $k \le (1- 1/t) d$.  Let $\gamma = |U|/n = \ceil{n/t}/n$, and note that $\gamma \le \frac{1}{t} + \frac{1}{n}$.  Since $k\le (1-1/t) d \le \frac{1}{1+t/n} d$, we may set $\delta = d/(t\gamma k) - 1$ so that $\delta \ge 0$ and $(1+\delta)\gamma k = d/t$.  Now, we compute that $\delta\gamma = \frac{d}{tk} - \gamma \ge \frac{d}{t(1-1/t)d} - \gamma = \frac{1}{t-1} - \gamma \ge \frac{1}{t-1} - (\frac{1}{t} + \frac{1}{n}) \ge \frac{1}{2t^2}$.  Therefore Lemma~\ref{lem:ci} shows that $q_k \le e^{-\frac{(\delta\gamma)^2}{2} k} \le e^{-\frac{1}{8t^4} k}$.  When $k \ge d/t$, this reduces to $q_k \le e^{-\frac{1}{8t^5} d}$.  Hence, we have that 
\begin{align*}
\hb(\HH) &= \sum_{k = 0}^n p_k\\
&\ge \left(\sum_{k< d/t} 1 \right) +  \left(\sum_{d/t \le k \le (1-1/t)d} 1 - tq_k \right)\\
&\ge (1-1/t)d - t\left(\sum_{d/t \le k \le (1-1/t)d} e^{-\frac{1}{8t^5} d}\right) \\
&\ge (1-1/t)d - tde^{-\frac{1}{8t^5} d}.
\end{align*}
For $n$ sufficiently large, $\G$ and $\HH$ are disjoint, and hence $\hb(\F) = \hb(\G) + \hb(\HH) \ge \floor{(1-1/t)d} + (1-1/t)d - tde^{-\frac{1}{8t^5} d}$.  Setting, for example, $t=d^{1/6}$, we have that $\hb(\F) \ge (2-o(1))d$ as $d\to\infty$.

\begin{lemma}\label{lem:univL}
Let $\gamma$ be a constant satisfying $0<\gamma<1$.  If $\F$ is a family of subsets of $[n]$ and $\hb(\F)>4r/\gamma+2\ln(2)/(1-\gamma) + 2$, then $\F$ contains $\univ{r}$ or $\F$ contains $\univd{r}$.
\end{lemma}
\begin{proof}
Let $\F$ be a family of subsets of $[n]$ with $\hb(\F)>4r/\gamma+2\ln(2)/(1-\gamma) + 2$.  We define $\overline{\F} = \{[n] - A\st A\in\F\}$.  Note that $\overline{\F}$ is the dual of $\F$ and $\hb(\F) = \hb(\overline{\F})$.  At least one of $\{A\in\F\st |A| \le n/2\}$ and $\{A\in\overline{\F}\st |A| \le n/2\}$ has Lubell mass at least $\hb(\F)/2$.  If the former occurs, then we obtain a copy of $\univd{r}$ in $\F$.  If the latter occurs, then we obtain a copy of $\univd{r}$ in $\overline{\F}$, which immediately yields a copy of $\univ{r}$ in $\F$.  Let $\F_1 = \{A\in\F\st |A|\le n/2\}$.  We may assume that $\hb(\F_1) \ge \hb(\F)/2$.

Let $\F_2$ be the set of all $A$ in $\F_1$ that are not $\gamma$-flexible, and let $\F_3$ be the set of all $A$ in $\F_1$ such that $\hb_A^-(\F_1) \le 2r/\gamma$.  By Lemma~\ref{lem:flex}, we have that $\hb(\F_2) < 1 + \ln 2/(1-\gamma)$.  By Corollary~\ref{cor:shallow}, we have that $\hb(\F_3) \le 2r/\gamma$.  Note that 
\[
\hb(\F_1) - (\hb(\F_2) + \hb(\F_3)) > \hb(\F)/2 - \left(1 + \frac{\ln 2}{1-\gamma} + \frac{2r}{\gamma} \right) > 0 
\]
Hence $\F_1 - (\F_2 \cup \F_3)$ is non-empty.

Choose $A\in\F_1 - (\F_2 \cup \F_3)$ and let $T$ be the set of all pivots in $A$.  Note that $|T|\ge \gamma|A|$ since $A\not\in \F_2$.  Let $\F'$ be the projection of the closed downset $D_{\F_1}[A]$ onto $T$.  Since $A\not\in \F_3$, we have that $\hb_A^-(\F_1) > 2r/\gamma$.  By Lemma~\ref{lem:proj}, we have that $\hb_T^-(\F') \ge \frac{|T| + 1}{|A| + 1}\hb_A^-(\F_1) > \gamma\hb_A^-(\F_1) \ge 2r$.  By Corollary~\ref{cor:VCdimL}, we have that $\vcdim(\F') \ge r$, and hence there is an $r$-set $R$ such that $R\subseteq T$ and $\F'$ shatters $R$.  For each $S\subseteq R$, there exists a set $B'_S\in\F'$ such that $B'_S \cap R = S$, and $B'_S$ extends to a set $B_S \in D_{\F_1}[A]$ such that $B_S \cap T = B'_S$.  

Note that $R\subseteq T$ and each element in $T$ is a pivot of $A$.  For each $i\in R$, find $A_i \in \F_1$ such that $A_i = (A - \{i\}) \cup \{j\}$ for some $j\in [n]$.  Since $B_S \subseteq A$, and $A - A_i = \{i\}$, we have that $B_S \subseteq A_i$ if and only if $i\not\in S$.  Therefore we obtain a copy of $\univd{r}$ in which the maximal elements are $\{A_i \st i\in R\}$ and the minimal elements are $\{B_S \st S\subseteq R\}$. 
\end{proof}

We are now able to prove our main theorem.

\begin{theorem}\label{thm:bip}
If $P$ is an $r$-element poset of height at most $2$, then $\slth(P) \le 4r + \sqrt{(32 \ln 2)r} + 6$.
\end{theorem}
\begin{proof}
With $c=\sqrt{(\ln 2) / 2} \approx 0.589$ and $\gamma = 1-c/\sqrt{r}$, some algebra shows that $4r + \sqrt{(32 \ln 2)r} + 6 > 4r/\gamma+2\ln(2)/(1-\gamma) + 2$.  Hence if $\F \subseteq \tsupn$ and $\hb(\F) > 4r + \sqrt{(32 \ln 2)r} + 6$, then it follows from Lemma~\ref{lem:univL} that $\F$ contains a copy of $\univ{r}$ or a copy of $\univd{r}$, and by Lemma~\ref{lem:univP} both of these contain a copy of $P$.
\end{proof}

We do not expect that Theorem~\ref{thm:bip} is sharp, and indeed, this technique sometimes gives better bounds when specialized to particular posets.  We examine one such family which shows that the coefficient on the linear term cannot be reduced beyond $1/2$.

\subsection{Lubell and Tur\'an Thresholds of Standard Examples}

The \emph{standard example}, denoted $\stdex{r}$, is a poset on $2r$ elements $a_1, \ldots, a_r$ and $b_1, \ldots, b_r$ such that $\{a_1,\ldots,a_r\}$ and $\{b_1,\ldots,b_r\}$ are antichains and $b_j \le a_i$ if and only if $i\ne j$.  Note that the dual of the standard example is itself.

To find copies of $S_r$, the full power of Corollary~\ref{cor:VCdimL} is not necessary.  Let $\F\subseteq \tsupn$ and let $R\subseteq [n]$.  We say that $\F$ contains an \emph{$R$-system of private elements} if $\F$ contains sets $(B_i)_{i\in R}$ such that for each $i,j \in R$, we have $i\in B_j$ if and only if $i=j$.  Of course, if $\F$ shatters $R$, then one easily obtains an $R$-system of private elements.  As our next lemma shows, if we need only a system of private elements, the required Lubell mass is reduced by a factor of 2.

\begin{lemma}\label{lem:priv}
If $\F\subseteq \tsupn$ and $\hb(\F)> r$, then there exists an $r$-set $R$ such that $\F$ contains an $R$-system of private elements.
\end{lemma}
\begin{proof}
By induction on $r + n$.  If there exists a proper subcube $[A,B]$ of $\tsupn$ with $\hb(\F, [A,B]) \ge \hb(\F)$, then we apply induction to the family $\F \cap [A,B]$ in $[A,B]$, and the obtained $R$-system of private elements in $[A,B]$ is valid in $\tsupn$.  Hence, we may assume that $\hb(\F, [A,B]) < \hb(\F)$ for each subcube $[A,B]$.  In particular, for each $i\in [n]$, there is a set $A\in\F$ containing $i$, as otherwise we would have $\hb(\F, Q_i) \ge \hb(\F)$, where $Q_i$ is the subcube $[\emptyset, [n] - \{i\}]$.

If $n=0$ or $r=0$, then the statement holds trivially, so suppose that $n\ge 1$ and $r\ge 1$.  Note that $\hb(\F) = \epsilon + \sum_{i\in [n]} \hb(\F, Q_i)$ where $\epsilon = 1$ if $[n] \in \F$ and $\epsilon = 0$ otherwise.  It follows that there exists $j\in [n]$ such that $\hb(\F, Q_j) \ge \hb(\F) - 1 > r -1 $.  Applied to the subcube $Q_j$, it follows by induction that $\F \cap Q_j$ contains an $R'$-system of private elements $(B_i)_{i\in R'}$ with $|R'| = r-1$.  Setting $R = R'\cup \{j\}$ and choosing $B_j \in \F$ so that $j\in B_j$, we obtain an $R$-system of private elements with $|R| = r$.
\end{proof}

Lemma~\ref{lem:priv} is sharp, since the family $\{\emptyset\} \cup \{A\subseteq [n]\st |A| > n-(r-1)\}$ has Lubell mass $r$ but does not contain an $R$-system of private elements with $|R| = r$.

\begin{lemma}\label{lem:stdex}
Let $\gamma$ and $\delta$ be constants satisfying $0<\gamma, \delta<1$.  If $\F \subseteq \{A \subseteq [n]\st |A| \le \delta n\}$ and $\hb(\F)>r/\gamma+\frac{1}{1-\gamma}\ln \frac{1}{1-\delta} + 1$, then $\F$ contains $\stdex{r}$.
\end{lemma}
\begin{proof}
The proof is the same as in Lemma~\ref{lem:univL} with slight modifications.  Let $\F_1$ be the set of all $A$ in $\F$ that are not $\gamma$-flexible, and let $\F_2$ be the set of all $A$ in $\F$ such that $\hb_A^-(\F) \le r/\gamma$.  By Lemma~\ref{lem:flex}, we have that $\hb(\F_1) < 1 + \frac{1}{1-\gamma}\ln \frac{1}{1-\delta}$.  By Corollary~\ref{cor:shallow}, we have that $\hb(\F_2) \le r/\gamma$.  Since $\hb(\F) - \hb(\F_1) - \hb(\F_2) > 0$, there exists $A\in\F - (\F_1 \cup \F_2)$.

Let $T$ be the set of all pivots in $A$.  Note that $|T|\ge \gamma|A|$ since $A\not\in \F_1$.  Let $\F'$ be the projection of the closed downset $D_{\F}[A]$ onto $T$.  Since $A\not\in \F_2$, we have that $\hb_A^-(\F) > r/\gamma$.  By Lemma~\ref{lem:proj}, we have that $\hb_T^-(\F') \ge \frac{|T| + 1}{|A| + 1}\hb_A^-(\F) > \gamma\hb_A^-(\F) \ge r$.  By Lemma~\ref{lem:priv}, there is an $R$-system of private elements $(B'_j)_{j\in R}$ in $\F'$.  Each $B'_j \in \F'$ extends to a set $B_j \in D_{\F}[A]$ such that $B_j \cap T = B'_j$.  

Note that $R\subseteq T$ and each element in $T$ is a pivot of $A$.  For each $i\in R$, find $A_i \in \F_1$ such that $A_i = (A - \{i\}) \cup \{j\}$ for some $j\in [n]$.  Since $B_j \subseteq A$ and $A - A_i = \{i\}$, we have that $B_j \subseteq A_i$ if and only if $i\not\in B_j$.  Hence $B_j \subseteq A_i$ if and only if $j\ne i$ and $\{A_i\st i\in R\}$ and $\{B_j \st j\in R\}$ form a copy of $\stdex{r}$.
\end{proof}

\begin{theorem}\label{thm:stdexL}
$r-2 \le \tth(\stdex{r}) \le \slth(\stdex{r}) \le 2r + \sqrt{(16 \ln 2)r} + 11$.
\end{theorem}
\begin{proof}
For the upper bound, let $\F$ be a subfamily of $\tsupn$ with $\hb(\F) > 2r + \sqrt{(16 \ln 2)r} + 11$, and define $\F^+ = \{A\in \F \st |A|\ge n/2\}$ and $\F^- = \{A\in \F \st |A| \le n/2\}$.  Note that $\hb(\F^+) \ge \hb(\F)/2$ or $\hb(\F^-) \ge \hb(\F)/2$.  Set $c=\sqrt{\ln 2} \approx 0.833$, $\gamma = 1-c/\sqrt{r}$, and $\delta = 1/2$.  Some algebra shows that $\hb(\F)/2 \ge r + \sqrt{(4 \ln 2)r} + 5.5 > r/\gamma+\frac{1}{1-\gamma}\ln\frac{1}{1-\delta} + 1$.  If $\hb(\F^-) \ge \hb(\F)/2$, then $\F^-$ contains a copy of $\stdex{r}$ by Lemma~\ref{lem:stdex}.  Otherwise, we set $\F' = \{[n] - A\st A\in\F^+\}$ and apply Lemma~\ref{lem:stdex} to obtain a copy of $\stdex{r}$ in $\F'$.  Since $\stdex{r}$ is self-dual, this implies that $\F^+$ also contains a copy of $\stdex{r}$.

For the lower bound, we may assume that $r\ge 2$.  We show that if $\F$ is the union of $r-2$ contiguous levels, then $\F$ does not contain a copy of $\stdex{r}$.  Choosing the largest $r-2$ levels, we obtain the lower bound.  Suppose that $A_1, \ldots, A_r$ and $B_1, \ldots, B_r$ are subsets of $[n]$ such that $B_j \subseteq A_i$ if and only $i\ne j$.  Note that for each $B_j$, there exists an element $b_j \in [n]$ such that $b_j \in B_j$ but $b_j\not\in B_k$ for $k\ne j$, as otherwise $B_j \subseteq \bigcup_{k\ne j} B_k \subseteq A_j$.  Since $B_1 \subseteq A_2$ and $\{b_3, \ldots, b_r\} \subseteq A_2 - B_1$, it follows that $|A_2| - |B_1| = |A_2 - B_1| \ge r-2$.  Thus, it is not possible for both $B_1, A_2 \in \F$ if $\F$ is the union of $r-2$ contiguous levels.
\end{proof}

Since $\stdex{r}$ is a $2r$-element poset of height $2$, the lower bound in Theorem~\ref{thm:stdexL} shows that the linear coefficient in Theorem~\ref{thm:bip} cannot be reduced beyond $1/2$.  While a large gap remains in our bounds on the Lubell threshold of standard examples, we are able to determine the Tur\'an threshold asymptotically using one additional observation.

\begin{lemma}\label{lem:tail}
If $\F$ is the family $\{A\subseteq [n]\st |A| \ge n/2 + \sqrt{2n\ln n}\}$, then $|\F| \le 2^n/n$.
\end{lemma}
\begin{proof}
Let $X$ be the size of a randomly chosen subset of $[n]$.  By Azuma's inequality, $\Pr[X \ge n/2 + t] \le e^{\frac{-t^2}{2n}}$.  It follows that $|\F| = 2^n \Pr[X \ge n/2 + \sqrt{2n\ln n}\}] \le 2^n/n$.  
\end{proof}

\begin{theorem}\label{thm:stdexT}
$r-2 \le \tth(\stdex{r}) \le r + 2\sqrt{r} + 6$
\end{theorem}
\begin{proof}
The lower bound is the same as in Theorem~\ref{thm:stdexL}.  For the upper bound, the case that $r=1$ is trivial since $\stdex{1}$ is a $2$-element antichain.  Suppose that $r\ge 2$ and that $n$ is sufficiently large so that $n/2 + \sqrt{2n\ln n} \le (1-1/e)n \approx 0.632 n$.  Let $\F$ be a subfamily of $\tsupn$ with $|\F| \ge (r+2\sqrt{r} + 6)\nchn$.  We show that $\F$ contains a copy of $\stdex{r}$.

With $\delta = 1-1/e$, let $\F' = \{A\in \F \st |A| \le \delta n\}$.  By Lemma~\ref{lem:tail}, we have that $|\F'| \ge |\F| - 2^n/n \ge (r + 2\sqrt{r} + 5)\nchn$ and therefore $\hb(\F') \ge r + 2\sqrt{r} + 5$.  Let $\gamma = 1-1/\sqrt{r}$ and note that $r + 2\sqrt{r} + 5 > r/\gamma+\frac{1}{1-\gamma}\ln \frac{1}{1-\delta} + 1$ for $r\ge 2$.  By Lemma~\ref{lem:stdex}, it follows that $\F'$ contains a copy of $\stdex{r}$.
\end{proof}

\section{Lubell and Tur\'an Thresholds of Small Boolean Lattices}\label{sec:dia}

Each $n$-element poset $P$ is contained in the subset lattice with ground set $P$.  Indeed, since $x\le y$ if and only if $D[x] \subseteq D[y]$, the subset lattice contains a copy of $P$ in which $x\in P$ is represented by $D[x]$.  Consequently, Conjecture~\ref{conj:tt} is equivalent to the assertion that every Boolean lattice has a finite Tur\'an threshold, and Conjecture~\ref{conj:lt} is equivalent to the assertion that every Boolean lattice has a finite Lubell threshold.  In this section, we establish a finite Lubell threshold for $\bool{3}$ and provide an preliminary bound on the Tur\'an threshold for $\bool{2}$.  It is open whether or not $\bool{4}$ has a finite Tur\'an or Lubell threshold.  Our next bound follows naturally from our results on standard examples.

\begin{cor}\label{cor:b3}
$\slth(\bool{3}) \le 16$.
\end{cor}
\begin{proof}
First, we show that $\slth(\bool{3}) = \slth(\stdex{3})$.  Since $\stdex{3}$ is a subposet of $\bool{3}$, it is immediate that $\slth(\stdex{3}) \le \slth(\bool{3})$.  Note that if $\F \subseteq \tsupn$ and $\F$ is $\stdex{3}$-free, then $\F \cup \{\emptyset, [n]\}$ is also $\stdex{3}$-free.  It follows that if $\F$ is an $\stdex{3}$-free family of subsets of $[n]$ and $\F$ contains neither $\emptyset$ nor $[n]$, then $\hb(\F) \le \slth(\stdex{3}) - 2$.  Let $\F$ be a $\bool{3}$-free family of subsets of $[n]$.  By Proposition~\ref{prop:UD}, there exist $A,B\in\F$ such that $A\subseteq B$ and $\hb(\F, I)\ge\hb(\F)$, where $I=[A,B]$.  Let $\F' = \{C - A\st C\in \F\cap I\}$, and note that $\F'$ is a subset of an $n'$-dimensional Boolean lattice, where $n' = |B|-|A|$.  Obtain $\F''$ from $\F'$ by removing the minimum and maximum elements.  Since $\F$ is $\bool{3}$-free, it follows that $\F''$ is $\stdex{3}$-free.  It follows that $\hb(\F) - 2 \le \hb(\F') - 2 = \hb(\F'') \le \slth(\stdex{3}) - 2$ and therefore $\hb(\F) \le \slth(\stdex{3})$.

It remains to show that $\slth(\stdex{3}) \le 16$.  The bound in Theorem~\ref{thm:stdexL} is optimized for large $r$; in the case $r=3$, we obtain a better bound directly from Lemma~\ref{lem:stdex}.  Let $\F \subseteq \tsupn$ with $\hb(\F)>16$.  At least one of $\{A\in \F \st |A| \le n/2\}$ and $\{A\in \F\st |A| \ge n/2\}$ has Lubell mass larger than $8$.  In either case, we may apply Lemma~\ref{lem:stdex} with $\delta = 1/2$ and $\gamma = 2/3$ to obtain a copy of $\stdex{3}$ in $\F$.  
\end{proof}

We make no effort to optimize the bound in Corollary~\ref{cor:b3} further.  From below, only the trivial bounds $\lth(\bool{3}) \ge \tth(\bool{3}) \ge \tth(\chain{4}) = 3$ are known.  It is conceivable that $\tth(\bool{3}) = 3$ or even $\lth(\bool{3}) = 3$.  (The strong Lubell threshold satisfies $\slth(\bool{3}) \ge 3+ 2/3$, since the family $\bool{3} - \{1\}$ is $\bool{3}$-free.)

Our next goal is to establish the strong Lubell threshold of $\bool{2}$ exactly.  Let $a_n = \sum_{k=0}^n 1/\binom{n}{k}$.  We present initial values of this sequence the table below.
\begin{center}
\begin{tabular}{c|*{9}{|c}}
$n$ & 0 & 1 & 2 & 3 & 4 & 5 & 6 & 7 & 8\\
\hline
$a_n$ & $1$ & $2$ & $2 + \frac{1}{2}$ & $2+\frac{2}{3}$ & $2+\frac{2}{3}$ & $2+\frac{3}{5}$ &  $2 + \frac{31}{60}$ & $2+\frac{46}{105}$ & $2+\frac{13}{35}$\\
approx. & 1 & 2 & 2.5 & $2.667$ & $2.667$ & $2.6$ & $2.517$ & $2.438$ & $2.371$\\
\end{tabular}
\end{center}
It is easy to show that $(a_n)_{n\ge 0}$ is a unimodal sequence that achieves its maximum value $2+2/3$ at $n=3$ and $n=4$.

\begin{prop}\label{prop:b2SL}
$\slth(\bool{2}) = 2 + 2/3$.
\end{prop}
\begin{proof}
Let $\F\subseteq \tsupn$ and suppose that $\F$ is $\bool{2}$-free.  By two applications of Proposition~\ref{prop:UD}, there exists a subcube $[A,B]$ with $A,B\in \F$ such that $\hb(\F, [A,B]) \ge \hb(\F)$.  Let $\F' = \{C - A\st C\in \F\cap [A,B]\}$ and note that $\F'$ is a subfamily of an $n'$-dimensional Boolean lattice, where $n' = |B| - |A|$.  Since $\F'$ contains its minimum and maximum elements and also $\F'$ is $\bool{2}$-free, it follows that $\F'$ is a chain.  Therefore $\F'$ contains at most one set of each size, which implies that $\hb(\F') \le a_{n'} \le 2+2/3$.  Hence $\hb(\F) \le \hb(\F, [A,B]) = \hb(\F') \le 2+2/3$, and therefore $\slth(\bool{2}) \le 2 + 2/3$.

For the lower bound, let $\F_n$ be a maximal chain in $\tsupn$.  Clearly, $\F_n$ is $\bool{2}$-free and $\hb(\F_n) = a_n$.  By choosing $n=3$ or $n=4$, we obtain the desired lower bound.
\end{proof}

\newcommand{\mlth}{{\rho^*}}
As we have seen, to prove upper bounds on the strong Lubell threshold of $P$, it suffices to consider $P$-free families $\F\subseteq \tsupn$ with $\emptyset, [n]\in \F$.  While the analogue for the Lubell threshold does not generally hold, a weaker variant does hold for the Tur\'an threshold, where we may assume that $\emptyset \in \F$.  Griggs, Li, and Lu~\cite{GLL} proved the analogous statement for the extension variant.  In the induced case, the proof is the same;  we include it for completeness.  Let $\mlth(P) = \limsup_{n\to\infty} \{\hb(\F)\st \mbox{$\emptyset \in \F \subseteq \tsupn$ and $\F$ is $P$-free}\}$.   

\begin{lemma}\label{lem:emptyin}
$\tth(P) \le \mlth(P) \le \lth(P)$.
\end{lemma}
\begin{proof}
Since an additional restriction is imposed on $P$-free families in the definition of $\mlth(P)$, it is immediate that $\mlth(P) \le \lth(P)$.

For the other inequality, let $(\F_n)_{n\ge 1}$ be a sequence of $P$-free families such that $\F_n \subseteq \tsupn$ and $\limsup_{n\to\infty} |\F|/\nchn = \tth(P)$.  Let $\F'_n$ be the subfamily of $\F_n$ consisting of all sets $A\in \F_n$ with $|A| \le n/2 + \sqrt{2n\ln n}$.  By Lemma~\ref{lem:tail}, we have that $|\F_n| - |\F'_n| \le 2^n/n = o(\nchn)$.  It follows that $\hb(\F'_n) \ge |\F'_n|/\nchn \ge |\F_n|/\nchn - o(1)$.  By Proposition~\ref{prop:UD}, there exists $A_n \in \F'_n$ such that $\hb(\F', I_n) \ge \hb(\F')$, where $I_n$ is the subcube $[A_n, [n]]$.  Let $\F''_n = \{C-A_n\st C\in \F'_n \cap I_n\}$, and note that $\emptyset \in \F''_n$ and $\F''_n$ is subfamily of an $n''$-dimensional cube with $n'' = n - |A| \ge n - (n/2 + \sqrt{2n\ln n})$.  Therefore $(\F''_n)_{n\ge 1}$ is a sequence of $P$-free families in Boolean lattices of unbounded dimension and $\hb(\F''_n) = \hb(\F'_n, I_n) \ge \hb(\F'_n) \ge |\F_n|/\nchn - o(1)$.  It follows that $\mlth(P) \ge \limsup_{n\to\infty} \hb(\F''_n) \ge \tth(P)$.  
\end{proof}

Using $\mlth(P)$, we improve the bound $\tth(\bool{2}) \le 2 + 2/3$ that follows from Proposition~\ref{prop:b2SL}.  This gives a preliminary upper bound on $\tth(\bool{2})$.  We expect that additional improvements are possible, and we hope this motivates additional work on the induced variant of the diamond-free poset problem.

\newcommand{\cun}{\vee}
\begin{theorem}\label{thm:b2ML}
$2.2818564< \mlth(\bool{2})<2.5823284$.
\end{theorem}
\begin{proof}
The lower bound is due to the following construction.
Partition $[n]$ into sets $S$, $T$, and $R$.  For families $\G$ and $\HH$, let $\G \cun \HH = \{A\cup B\st \mbox{$A\in\G$ and $B\in\HH$}\}$.  Let $\F$ be the set family that is the union of the following subfamilies:
\begin{enumerate}
\item $\{\emptyset\}$,
\item $\binom{S}{1}$,
\item $\binom{S}{1}\cun \binom{T}{1}$,
\item $\binom{S}{1}\cun \binom{T}{1}\cun \binom{R}{1}$,
\item $\binom{T}{2}$,
\item $\binom{R}{2}$, 
\item $\binom{T}{2} \cun \binom{R}{1}$, and
\item $\binom{T}{1} \cun \binom{R}{2}$.
\end{enumerate}
It is clear that $\F$ contains no induced copy of $\bool{2}$.
Set $|S|=x_1n$, $|T|=x_2n$, and $|R|=x_3n$, where $x_1+x_2+x_3=1$.
We have
\begin{align*}
\Lubell(\F) &=1+\frac{|S|}{n}+ \frac{|S||T|+{|T|\choose
    2}+{|R|\choose 2}}{{n\choose
    2}}
+ \frac{|S||T||R|+|T|{|R|\choose 2}| +{|T|\choose 2}|R|}{{n\choose
    3}}\\
  &= 1+x_{{1}}+2\,x_{{1}}x_{{2}}+6\,x_{{1}}x_{{2}}x_{{3}}+{x_{{2}}}^{2}+3\,
{x_{{2}}}^{2}x_{{3}}+3\,x_{{2}}{x_{{3}}}^{2}+{x_{{3}}}^{2}+O(1/n).
\end{align*}
This polynomial achieves the maximum value
${\frac {428}{243}}+{\frac {40}{243}}\,\sqrt {10}\approx 2.281856404$
at 
$x_1=x_3=\frac{\sqrt{10}-1}{9}\approx .2402530734$ and
$x_2=\frac{11+2\sqrt{10}}{9}\approx  .5194938532$.
Hence, $ \mlth(\bool{2})>2.2818564$.

The upper bound follows from the first moment method.  Let $\F$ be a $\bool{2}$-free subset of $\tsupn$ with $\emptyset\in \F$.  Note that for each $A\in\F$, we have that $\F \cap [\emptyset, A]$ is a chain, and therefore $\F \cap [\emptyset, A]$ contains at most one set of each size in $\{0, \ldots, |A|\}$.  Let $\G$ be the subfamily of all $A
\in \F$ such that $\F \cap [\emptyset, A]$ does not contain a set of size $1$.  For $k\ge 1$, let $\M_k = (\F - \G) \cap \binom{[n]}{k}$.  Note that $\G$ and $(\M_k)_{k\ge 1}$ form a partition of $\F$.

Choose a maximal chain $\C$ in $\tsupn$ uniformly at random.  For each $A\in \F$, let $\C_A$ be the event that $A$ is the maximum set in $\F \cap \C$.   Then we have
\begin{align}
\Lubell(\F)&=\sum_{A\in \F} \Lubell_A^-(\F)\Pr(\C_A)\\
&= \sum_{A\in \G} \Lubell_A^-(\F)\Pr(\C_A) + \sum_{k \ge 1} \left( \sum_{A\in \M_k} \Lubell_A^-(\F)\Pr(\C_A) \right)
\end{align}
\newcommand{\EE}{\mathcal{E}}%
Since $\F \cap [\emptyset, A]$ contains at most one set of each size for each $A\in \F$, it follows that $\Lubell_A^-(\F) \le a_k$, where $k=|A|$.  Moreover, if $A\in\G$, then $\Lubell_A^-(\F) \le a_k - 1/k$.  It is easy to check that the sequence $(a_k-1/k)_{k\ge 1}$ achieves its maximum value $2+5/12$ at $k=4$.  Let $\alpha = 2+5/12$.  Setting $\EE_0$ to be the event that the maximum set in $\F \cap \C$ belongs to $\G$ and $\EE_k$ to be the event that the maximum set in $\F \cap \C$ belongs to $\M_k$, it follows that 
\begin{align}
\Lubell(\F)&\le \sum_{A\in \G} \alpha\Pr(\C_A) + \sum_{k \ge 1} \left( \sum_{A\in \M_k} a_k\Pr(\C_A) \right)\\
\label{eq:tosort}
&= \alpha\Pr(\EE_0) + \sum_{k\ge 1} a_k \Pr(\EE_k).
\end{align}
Note that the sequence $(a_k)_{k\ge 0}$ satisfies $a_3 = a_4 > a_5 > a_6 > a_2 > a_7 > \alpha$ and $a_k < \alpha$ when $k\le 1$ or $k\ge 8$.  With $\EE_{i,j} = \bigcup_{i\le k\le j} \EE_k$, we rewrite Equation~\ref{eq:tosort}:
\begin{align*}
\Lubell(\F) \le ~&(a_3-a_5) \Pr(\EE_{3,4}) + 
(a_5 - a_6) \Pr(\EE_{3,5}) + 
(a_6 - a_2) \Pr(\EE_{3,6}) + ~\\
& (a_2 - a_7) \Pr(\EE_{2,6}) + (a_7 - \alpha) \Pr(\EE_{2,7}) + \alpha.
\end{align*}
Let $S = \{s\st \{s\}\in \F\}$, let $T = [n] - S$, and let $x=|T|/n$.  Recall that $\C$ is a random maximal chain in $\tsupn$.  If the event $\EE_{i,j}$ occurs, then the maximum set $A$ in $\F \cap \C$ belongs to $\bigcup_{i\le k\le j} \M_k$.  For each $k$, let $C_k$ be the unique set in $\C$ of size $k$.  Since $i \le |A| \le j$, it follows that $C_i \subseteq A \subseteq C_j$.  Since $|A \cap S| = 1$, it is not possible that $|C_i \cap S| \ge 2$.  Nor is it possible that $|C_j \cap S| = 0$.  As these are disjoint events, we have that 
\begin{align*}
\Pr(\EE_{i,j}) &\le 1 - \Pr(|C_i \cap S| \ge 2) - \Pr(|C_j \cap S| = 0)\\
&\le 1 - \left(1 - \Pr(|C_i \cap S| = 0) - \Pr(|C_i \cap S| = 1)\right) - \Pr(|C_j \cap S| = 0)\\
&\le \Pr(|C_i \cap S| = 0) + \Pr(|C_i \cap S| = 1) - \Pr(|C_j \cap S| = 0).
\end{align*}
When $i$ and $j$ are fixed and $n$ is large, this becomes
\begin{align*}
\Pr(\EE_{i,j}) \le x^i + i(1-x)x^{i-1} - x^j + O(1/n).
\end{align*}
Substituting into our bound and simplifying shows that
\begin{align*}
\Lubell(\F) \le \frac{29}{12} + \frac{1}{6}x + \frac{5}{12}x^2 - 
\frac{1}{3} x^3 - \frac{1}{15}x^4 - \frac{1}{12}x^5 - 
\frac{11}{140}x^6 - \frac{3}{140}x^7 + O(1/n).
\end{align*}
Using calculus, it is easy to check that the maximum value of the above
polynomial for $x\in[0,1]$ is $2.5823283024\cdots$, which is reached at
$x=0.6870021578\cdots$.  Thus, $\mlth(\bool{2})\leq 2.5823284$.
\end{proof}

For the extension variant, Kramer, Martin, and Young~\cite{KMY} proved that $\La(n,\bool{2})\leq (2.25+o(1))\nchn$.  In fact, they prove that if $\emptyset \in \F \subseteq \tsupn$ and $\F$ does not contain an extension of $\bool{2}$, then $\hb(\F) \le 2.25+o(1)$.  Since there are $\bool{2}$-free families with Lubell mass $2.25 - o(1)$, it is not possible to improve the upper bound $\La(n,\bool{2})\le (2.25+o(1))\nchn$ using Lubell-type arguments exclusively.  

The lower bound in Theorem~\ref{thm:b2ML} shows that the Lubell threshold of $\bool{2}$ is strictly larger in the induced variant than it is in the extension variant, providing additional evidence that the two problems are fundamentally different.  Just as known constructions show that $2.25$ is a natural barrier in the extension variant, the construction in Theorem~\ref{thm:b2ML} shows that one cannot improve the bound $\tth(\bool{2}) < 2.583$ beyond $2.28$ using Lubell-type arguments exclusively.

\bibliographystyle{abbrv}
\bibliography{forbidden-ind}

\end{document}